\documentclass[a4paper]{article}

\usepackage[english]{babel}
\usepackage[utf8x]{inputenc}
\usepackage{amsmath}
\usepackage{bm}

\usepackage[a4paper,top=3cm,bottom=2cm,left=3cm,right=3cm,marginparwidth=1.75cm]{geometry}

\usepackage{amsthm}
\usepackage{amsfonts}
\usepackage{graphicx}
\usepackage[colorlinks=true, allcolors=blue]{hyperref}
\usepackage{url}

\newtheorem{conjecture}{Conjecture}
\newtheorem{corollary}{Corollary}
\newtheorem{definition}{Definition}
\newtheorem{lemma}{Lemma}
\newtheorem{proposition}{Proposition}
\newtheorem{theorem}{Theorem}
\newtheorem{remark}{Remark}
\newtheorem{example}{Example}

\def\NN{\mathbb{N}}
\def\QQ{\mathbb{Q}}
\def\RR{\mathbb{R}}
\def\ZZ{\mathbb{Z}}

\title{On the finiteness and periodicity of the $p$--adic Jacobi--Perron algorithm}

\author{Nadir Murru and Lea Terracini \\
Department of Mathematics G. Peano, University of Torino\\
Via Carlo Alberto 10, 10123, Torino, ITALY\\
nadir.murru@unito.it, lea.terracini@unito.it}
\date{}

\begin{document}
\maketitle

\begin{abstract}
Multidimensional continued fractions (MCFs) were introduced by Jacobi and Perron in order to obtain periodic representations for algebraic irrationals, as it is for continued fractions and quadratic irrationals. Since continued fractions have been also studied in the field of $p$--adic numbers $\mathbb Q_p$, also MCFs have been recently introduced in $\mathbb Q_p$ together to a $p$--adic Jacobi--Perron algorithm. In this paper, we address th study of two main features of this algorithm, i.e., finiteness and periodicity. In particular, regarding the finiteness of the $p$--adic Jacobi--Perron algorithm our results are obtained by exploiting properties of some auxiliary integer sequences.
Moreover, it is known that a finite $p$--adic MCF represents $\mathbb Q$--linearly dependent numbers. We see that the viceversa is not always true and we prove that in this case infinite partial quotients of the MCF have $p$--adic valuations equal to $-1$. 
Finally, we show that a periodic MCF of dimension $m$ converges to algebraic irrationals of degree less or equal than $m+1$ and for the case $m=2$ we are able to give some more detailed results.
\end{abstract}


\section{Introduction}

Multidimensional continued fractions (MCFs) were introduced by Jacobi \cite{Jac} in order to answer to a question posed by Hermite \cite{Her}, namely the existence of an algorithm defined over the real numbers that becomes eventually periodic when it processes algebraic irrationalities. In other words, Hermite asked for a generalization of the classical continued fraction algorithm that produces a periodic expansion if and only if the input is a quadratic irrational. The Jacobi algorithm deals with cubic irrationals and it was generalized to higher dimensions by Perron \cite{Per}. However, the Jacobi--Perron algorithm does not solve the Hermite problem, since it has never been proved that it becomes eventually periodic when processing algebraic irrationals. Many studies have been conducted on MCFs and their modifications, see, e.g., \cite{Ass}, \cite{Bea}, \cite{Ber}, \cite{Gar}, \cite{Ger}, \cite{Hen}, \cite{Kar}, \cite{Mur1}, \cite{Mur2}, \cite{Sch1}, \cite{Sch2}, \cite{Tam}.

In the 1970s, some authors started to study one--dimensional continued fractions over the $p$--adic numbers \cite{Bro1}, \cite{Rub}, \cite{Schn}. From these studies, it appeared difficult to find an algorithm working on the $p$--adic numbers that produces continued fractions having the same properties holding true over the real numbers (regarding approximation, finiteness and periodicity). In particular, no algorithm which provides a periodic expansion for all quadratic irrationalities has been found. Continued fractions over the $p$--adic numbers have been also studied  recently in several works, like \cite{Bed1}, \cite{Bed2}, \cite{Bro2}, \cite{Cap}, \cite{Han}, \cite{Lao}, \cite{Mil}, \cite{Oot}, \cite{Poo}, \cite{Til}, \cite{Weg}.

Motivated by the above researches, in \cite{MT}, the authors started the study of MCFs in $\mathbb Q_p$, providing some results about convergence and finiteness. In particular, they gave a sufficient condition on  the partial quotients of a MCF that ensures the convergence in $\mathbb Q_p$. Moreover, they presented an algorithm that terminates in a finite number of steps when rational numbers are processed. 
The scope of that work was introducing the subject and providing some general properties; the terminating input of this algorithm was not fully characterized and the periodicity properties were not studied at all. This paper represents a continuation of the previous work, extending the investigation in these two directions. In particular, in Section \ref{sec:pre}, we fix the notation and we show some properties that can also be of general interest for MCFs. Section \ref{sec:fin} is devoted to the finiteness of the $p$-adic Jacobi--Perron algorithm, providing some results that improves a previous work \cite{MT} and showing also some differences with the real case. In $\mathbb R$, it is known that the Jacobi--Perron of dimension $2$ detects rational dependence, i.e., it terminates in a finite number of steps if and only if it processes rational linearly dependent inputs. On the contrary we show that this is not always true in $\mathbb Q_p$ and we also prove that in this case infinite partial quotients of the MCF have $p$--adic valuations equal to $-1$. Moreover, we give a condition that ensures the finiteness of the $p$--adic Jacobi--Perron algorithm in any dimension in terms of the $p$--adic valuation of the partial quotient. In Section \ref{sec:per}, we study the periodicity of MCF in $\mathbb Q_p$. 
Specifically, we introduce the characteristic polynomial related to a purely periodic $p$--adic MCF and we see that, as in the real case, it admits a $p$--adic dominant root which generates a field containing the limits of the MCF.
Consequently, we see that a periodic MCF of dimension $m$ converges to algebraic irrationalities of degree less or equal than $m+1$ as in the real case. A further investigation on the characteristic polynomial allows to characterize some cases where the degree is maximum.
We conclude our work with a conjecture which, if proved to be true, would give a characterization of the MCFs arising by applying the $p$-adic Jacobi-Perron algorithm to $m$-tuples consisting of $\QQ$-linear dependent numbers.

\section{Preliminaries and notation} \label{sec:pre}

The classical Jacobi--Perron algorithm processes a $m$--tuple of real numbers $\bm{\alpha}_0 = (\alpha_0^{(1)}, \ldots, \alpha_0^{(m)})$ and represents them by means of a $m$--tuple of integer sequences $(\mathbf{a}^{(1)}, \ldots, \mathbf{a}^{(m)}) = ((a_n^{(1)})_{n\geq 0}, \ldots, (a_n^{(m)})_{n\geq 0})$ (finite or infinite) determined by the following iterative equations:
\begin{equation*}
\begin{cases} a_n^{(i)} = [\alpha_n^{(i)}], \quad i = 1, ..., m, \cr
\alpha_{n+1}^{(1)} = \cfrac{1}{\alpha_n^{(m)} - a_n^{(m)}}, \cr
\alpha_{n+1}^{(i)} = \cfrac{\alpha_n^{(i-1)} - a_n^{(i-1)}}{\alpha_n^{(m)} - a_n^{(m)}}, \quad i = 2, ..., m, \end{cases} n = 0, 1, 2, ...
\end{equation*}
The integer numbers $a_n^{(i)}$ and the real numbers $\alpha_n^{(i)}$, for $i = 1, \ldots m$ and $n = 0, 1, \ldots$, are called \emph{partial quotients} and \emph{complete quotients}, respectively. The sequences of the partial quotients represent the starting vector $\bm{\alpha}_0$ by means of the equations
\begin{equation} \label{eq:MCF}\begin{cases} \alpha_n^{(i-1)} = a_n^{(i-1)} + \cfrac{\alpha_{n+1}^{(i)}}{\alpha_{n+1}^{(1)}}, \quad i = 2, ..., m \cr
\alpha_n^{(m)} = a_n^{(m)} + \cfrac{1}{\alpha_{n+1}^{(i)}} \end{cases} n = 0, 1, 2, ...\end{equation}
which produce objects that generalize the classical continued fractions and  are usually called \emph{multidimensional continued fractions} (MCFs).

The Jacobi--Perron algorithm has been translated into the $p$--adic field in \cite{MT}, using the function $s$ defined below that will play the role of the floor function. We define the set
$$\mathcal{Y}=\ZZ\left [\frac 1 p\right ]\cap \left (-\frac p2,\frac p 2\right).$$

\begin{definition}
The \emph{Browkin $s$-function} $s:\QQ_p\longrightarrow \mathcal{Y}$ is defined by
$$s(\alpha)= \sum_{j=k}^0 x_jp^j,$$
for every $\alpha\in\QQ_p$ written as
\(\alpha=\sum_{j=k}^\infty x_jp^j, \hbox{with } k, x_j\in\ZZ \hbox{ and } x_j\in \left (-\frac p 2,\frac p 2\right).\)
\end{definition}

Hence, the $p$--adic Jacobi--Perron algorithm processes a $m$--tuple of $p$--adic numbers $\bm{\alpha}_0 = (\alpha_0^{(1)}, \ldots, \alpha_0^{(m)})$ by the following iterative equations
\begin{equation} \label{eq:alg} \begin{cases}  a_n^{(i)} = s(\alpha_n^{(i)}) \cr 
\alpha_{n+1}^{(1)} = \cfrac{1}{\alpha_n^{(m)} - a_n^{(m)}} \cr 
\alpha_{n+1}^{(i)} = \alpha_{n+1}^{(1)}\cdot (\alpha_n^{(i-1)} - a_n^{(i-1)}) = \cfrac{\alpha_n^{(i-1)} - a_n^{(i-1)}}{\alpha_n^{(m)} - a_n^{(m)}}, \quad i = 2, ..., m
\end{cases}
\end{equation}
for $n = 0, 1, 2, \ldots$, which define a $p$--adic MCF $[(a_0^{(1)}, a_1^{(1)}, \ldots), \ldots, (a_0^{(m)}, a_1^{(m)}, \ldots)]$ representing the starting $m$--tuple $\bm{\alpha}_0$ in the following way:
\begin{equation*}
\alpha_n^{(i)} = a_n^{(i)} + \cfrac{\alpha_{n+1}^{(i+1)}}{\alpha_{n+1}^{(1)}}, \quad \alpha_{n}^{(m)}= a_n^{(m)}  + \cfrac{1}{\alpha_{n+1}^{(1)}}
\end{equation*}
for $i = 1, \ldots, m-1$ and any $n \geq 0$. The partial quotients satisfy the following conditions:
\begin{equation} \label{eq:conv}
\begin{cases}
\lvert a_n^{(1)} \rvert > 1 \cr
\lvert a_n^{(i)} \rvert < \lvert a_n^{(1)} \rvert, \quad i = 2, \ldots, m
\end{cases}
\end{equation}
for any $n \geq 1$, where in the following $\lvert \cdot \rvert$ will always denote the $p$--adic norm. Moreover, for any $n \geq 1$, we have
\begin{align}  & \lvert a_n^{(1)} \rvert = \lvert \alpha_n^{(1)} \rvert,
\hbox{ and for $i = 2, \ldots, m$}\nonumber\\
& \lvert a_n^{(i)} \rvert =
\left\{ 
\begin{array}{l} \lvert \alpha_n^{(i)}\rvert
\hbox{ if $\lvert \alpha_n^{(i)} \rvert\geq 1$}\\
0 \hbox{ if $\lvert \alpha_n^{(i)} \rvert < 1$}\end{array}
\right .\label{eq:alpha-norme}\\
& \lvert \alpha_n^{(i)} \rvert < \lvert \alpha_n^{(1)} \rvert \quad . \nonumber
\end{align}
\begin{remark}
In \cite{MT}, the authors showed that equations \eqref{eq:conv} ensure the convergence of a MCF in $\mathbb Q_p$, i.e., given a sequence of partial quotients satisfying \eqref{eq:conv} (even if they are not obtained by a specific algorithm), then the corresponding MCF converges to a $m$--tuple of $p$--adic numbers.  
\end{remark}
Similarly to the real case, we have the $n$--\emph{th convergents} of a multidimensional continued fraction defined by
$$
Q^{(i)}_n=\frac {A^{(i)}_{n}}{A^{(m+1)}_{n}},
$$
for $i=1,\ldots, m$ and $n\in\NN$, where
\begin{equation} \label{eq:nd-conv} 
A^{(i)}_{-j} = \delta_{ij}, \quad A^{(i)}_{0} = a^{(i)}_0, \quad A^{(i)}_{n} = \sum_{j=1}^{m}a^{(j)}_{n}A^{(i)}_{n-j} + A_{n-m-1}^{(i)}
\end{equation}
for $i = 1, \ldots, m + 1$, $j = 1, \ldots, m$ and any $n \geq 1$, where $\delta_{ij}$ is the Kronecker delta. It can be proved by induction that for every $n \geq 1$ and $i = 1, \ldots, m$, we have
\begin{equation} \label{eq:alpha0}
\alpha_0^{(i)}=\frac {\alpha_n^{(1)}A^{(i)}_{n-1}+ \alpha_n^{(2)}A^{(i)}_{n-2}+\ldots +\alpha_n^{(m+1)}A^{(i)}_{n-m-1} }{\alpha_n^{(1)}A^{(m+1)}_{n-1}+ \alpha_n^{(2)}A^{(m+1)}_{n-2}+\ldots +\alpha_n^{(m+1)}A^{(m+1)}_{n-m-1} }
\end{equation} 
We can also use the following matrices for evaluating numerators and denominators of the convergents:
\begin{equation}\label{eq:matriciA}
 \mathcal{A}_n = \begin{pmatrix} a_n^{(1)} &1 &0&\ldots &0\\
  a_n^{(2)} &0 &1&\ldots &0\\
   \vdots& \vdots& \vdots& \vdots& \vdots\\
   a_n^{(m)} &0 &0&\ldots &1\\  1 &0 &0&\ldots &0\end{pmatrix}
\end{equation}
for any $n \geq 0$. Indeed, if we put
\begin{equation*}
\mathcal{B}_n = \begin{pmatrix} {A^{(1)}_{n}} &{A^{(1)}_{n-1}} &\ldots & {A^{(1)}_{n-m}}\\
 {A^{(2)}_{n}} &{A^{(2)}_{n-1}} &\ldots & {A^{(2)}_{n-m}}\\
  \vdots &\vdots&\vdots& \vdots\\
  {A^{(m+1)}_{n}} &{A^{(m+1)}_{n-1}} &\ldots & {A^{(m+1)}_{n-m}}\end{pmatrix} 
\end{equation*}
we have
\begin{equation*}
\mathcal{B}_n = \mathcal{B}_{n-1}\mathcal{A}_n = \mathcal{A}_0\mathcal{A}_1\ldots \mathcal{A}_n, \quad \det \mathcal{B}_n = (-1)^{m(n+1)}.
\end{equation*} 
We also recall some properties proved in \cite{MT}.
\begin{proposition}
With the notation above, we have
\begin{equation*}   |A_n^{(m+1)}|= \prod_{h=1}^n |a^{(1)}_h|.\end{equation*}
for any $n\geq 1$.
\end{proposition}

\begin{proposition} \label{prop:V}
Given the sequences $(V_n)_{n\geq-m}$, $i = 1, \ldots, m$, defined by
$$V_n^{(i)} = A_n^{(i)} - \alpha_0^{(i)} A_n^{(m+1)}$$
we have
\begin{enumerate}
\item $\displaystyle \lim_{n \rightarrow +\infty} \lvert V_n^{(i)} \rvert = 0$
\item $V^{(i)}_n=\sum_{j=1}^{m+1}a_n^{(j)}V^{(i)}_{n-j}$
\item $\sum_{j=1}^{m+1} \alpha^{(j)}_{n}V^{(i)}_{n-j}=0.$
\end{enumerate}
\end{proposition}

Finally, we prove the following propositions that will be useful in the next sections.

\begin{proposition} \label{prop:sum-prod}
For $n\geq 1$, we have
$$\sum_{i=1}^{m+1}\alpha^{(i)}_n A^{(m+1)}_{n-i}=\prod_{j=1}^n \alpha^{(1)}_j.$$
\end{proposition}
\begin{proof} We proceed by induction on $n$. If $n=1$ the left-hand side is equal to $\alpha^{(1)}_1 A^{(m+1)}_{0}= \alpha^{(1)}_1$. For $n>1$ we can use the inductive hypothesis and write
\begin{eqnarray*}
\prod_{j=1}^{n-1} \alpha^{(1)}_j&=& \sum_{i=1}^{m+1}\alpha^{(i)}_{n-1}A^{(m+1)}_{n-1-i}\\
&=&A^{(m+1)}_{n-m-2} + \sum_{i=1}^{m}\alpha^{(i)}_{n-1}A^{(m+1)}_{n-1-i}\\ &=&A^{(m+1)}_{n-m-2} + \sum_{i=1}^{m}\left (a^{(i)}_{n-1}+\frac{\alpha^{(i+1)}_n}{\alpha^{(1)}_n}\right )A^{(m+1)}_{n-1-i} \\ &=& (\sum_{i=1}^{m}a^{(i)}_{n-1}A^{(m+1)}_{n-1-i} + A^{(m+1)}_{n-m-2} ) + \sum_{i=1}^{m}\left (\frac{\alpha^{(i+1)}_n}{\alpha^{(1)}_n}\right )A^{(m+1)}_{n-1-i}\\
&=& A_{n-1}^{(m+1)} + \sum_{i=1}^{m}\left (\frac{\alpha^{(i+1)}_n}{\alpha^{(1)}_n}\right )A^{(m+1)}_{n-1-i}\\
&=& \frac 1 {\alpha^{(1)}_n} \left( \alpha^{(1)}_nA_{n-1}^{(m+1)} + \sum_{i=1}^{m}\alpha^{(i+1)}_nA^{(m+1)}_{n-1-i} \right) \\
&=& \frac 1 {\alpha^{(1)}_n} \sum_{i=1}^{m+1}\alpha^{(i)}_nA^{(m+1)}_{n-i}. \end{eqnarray*}
proving the claim.
\end{proof}

\begin{proposition}\label{prop:boundA}
For $i=1,\ldots, m$ and $n\in\NN$ we have
$$|A^{(i)}_n|_\infty< \frac{p^{n+1}} 2,$$
where $|\cdot|_\infty$ denotes the Euclidean norm.
\end{proposition}
\begin{proof}
We prove the thesis by induction. For $n=0$, 
$$	|A^{(i)}_0|_\infty =|a^{(i)}_0|_\infty <\frac p 2;$$
for $n\leq m$
$$A^{(i)}_n= a^{(1)}_nA^{(i)}_{n-1}+\ldots + a^{(n)}_nA^{(i)}_0+a^{(n-i)}$$
By induction hypothesis, and since $|a^{(i)}_k|_\infty <\frac p 2$ for every $k$, 
$$|A^{(i)}_n|_\infty < \frac {p^{n+1}} 4+\frac {p^{n}} 4+\ldots +\frac{p^2} 4 + \frac p 2=\frac {p^2} 4\left (\frac {p^{n}-1}{p-1}\right )+\frac p 2< \frac {p^{n+1}} 2.$$
For $n>m$ , we have 
$$A^{(i)}_n= a^{(1)}_nA^{(i)}_{n-1}+\ldots + a^{(m)}_nA^{(i)}_{n-m}+A^{(i)}_{n-m-1}.$$
Again by induction hypothesis, and since $|a^{(i)}_k|_\infty <\frac p 2$ for every $k$, 
$$|A^{(i)}_n|_\infty < \frac {p^{n+1}} 4+\frac {p^{n}} 4+\ldots +\frac{p^{n-m+2}} 4 + \frac{p^{n-m}} 2=\frac {p^{n-m+2}} 4\left (\frac {p^{m}-1}{p-1}\right )+ \frac{p^{n-m}} 2< \frac {p^{n+1}} 2.$$
\end{proof}

\begin{proposition} \label{prop:minors}
Given the MCF $[(a_0^{(1)}, a_1^{(1)}, \ldots), \ldots, (a_0^{(m)}, a_1^{(m)}, \ldots)]$ 
\begin{itemize}
\item[a)] every minor of $\mathcal{B}_n$ is a polynomial  in $\ZZ[a^{(i)}_j, i=1,\ldots , m,\  j=0,\ldots n]$ and each monomial has the form $$\lambda c_0c_1\ldots c_{n}$$ where $\lambda\in\ZZ$ and $c_j=1$ or $c_j=a^{(i)}_j$ for some $i=1,\ldots m$.
 \item[b)]  The summand $\lambda a^{(1)}_0\ldots a^{(1)}_n$ does not appear in any principal minor of $\mathcal{B}_n$ except for the $1\times  1$ minor obtained by removing all rows and columns indexed by $2,\ldots, m+1 $; in this case $\lambda=\pm 1$.
\end{itemize}
\end{proposition}
\begin{proof}
\ \\
$a)$
We prove the thesis by induction on $n$. For $n=0$, we have $\mathcal B_n = \mathcal A_0$ and the thesis immediately follows. Suppose now that the statement holds for $n$ and consider $\mathcal{B}_{n+1}$. Let $M$ be a square submatrix of $\mathcal{B}_{n+1}$. If $M=\mathcal{B}_{n+1}$ then $\det(M)=\pm 1$ and we are done. So we suppose that some rows and columns miss in $M$. If $M$ does not contain the first column, then $M$ is a square submatrix of $\mathcal{B}_n$ and the result holds by inductive hypothesis. Therefore we suppose that $M$ contains the first column of $\mathcal{B}_{n+1}$ which is
         \begin{equation}\label{eq:primacolonna}\begin{pmatrix} A^{(1)}_{n+1}\\ A^{(2)}_{n+1}\\
         \vdots \\
         A^{(m+1)}_{n+1}\end{pmatrix} =
         \begin{pmatrix} a^{(1)}_{n+1}A^{(1)}_n+a^{(2)}_{n+1}A^{(1)}_{n-1}+\ldots + a^{(m)}_{n+1} A^{(1)}_{n-m+1}+A^{(1)}_{n-m}\\
         a^{(1)}_{n+1}A^{(2)}_n+a^{(2)}_{n+1}A^{(2)}_{n-1}+\ldots + a^{(m)}_{n+1} A^{(2)}_{n-m+1}+A^{(2)}_{n-m}\\
         \vdots\\
         a^{(1)}_{n+1}A^{(m+1)}_n+a^{(2)}_{n+1}A^{(m+1)}_{n-1}+\ldots +a^{(m)}_{n+1} A^{(m+1)}_{n-m+1}+A^{(m+1)}_{n-m}\end{pmatrix}
         \end{equation}
By the properties of the determinant, $\det(M)$ is the sum for $i=1,\ldots, m+1$ of the determinants of all matrices $M_i$ where $M_i$ is obtained from $M$ by replacing the first column by a  subvector of 
         $$ a^{(i)}_{n+1} \begin{pmatrix} A^{(1)}_{n+1-i}\\
         A^{(2)}_{n+1-i}\\
         \vdots \\
         A^{(m+1)}_{n+1-i}\end{pmatrix}$$
(to get an uniform notation, we put $a^{(m+1)}_k=1$, for every $k\in\NN$ ). Then we see that either two columns of $M_i$ are proportional, so that $\det(M_i)=0$, or $\det(M_i)=\pm a^{(i)}_{n+1}\det(M'_i)$ where $M'_i$ is a submatrix of $\mathcal{B}_n$. Then the claim holds by inductive hypothesis.\\
$b)$ Let $M$ be the square submatrix obtained from $\mathcal{B}_n$  by removing all rows and columns indexed by $I\subseteq\{1,\ldots, m+1\}$, and suppose that the summand   $\lambda a^{(1)}_0\ldots a^{(1)}_n$ appears in $M$. Then by $a)$, $M$ must contain the first column of $\mathcal{B}_n$, so that it must contain also the first row. Moreover, since $\det(\mathcal{B}_n)=\pm 1$, at least one row and the corresponding column are missing.  We argue again by induction on $n$. If $n=0$, then the last row must miss, (otherwise $\det(M)\in\{1,0\}$) so that the last column too must miss; then the row indexed by $m$ has the form $(a_0^{(m)},0,\ldots, 0)$ and this implied that it must miss, unless $m=1$, so that column $m$ is missing and so on. It follows that $I={2,\ldots, m+1}$, $\lambda=1$. Now suppose that the result holds for $\mathcal{B}_n$. The first column of $\mathcal{B}_{n+1}$ being as in \eqref{eq:primacolonna}, wee deduce by  $a)$ that $\lambda a^{(1)}_0\ldots a^{(1)}_n$ must be a summand of $\det(M_1)$, where $M_1$  is obtained from $M$ by replacing the first column by a  subvector of 
         $$ a^{(1)}_{n+1} \begin{pmatrix} A^{(1)}_{n}\\
         A^{(2)}_{n}\\
         \vdots \\
         A^{(m+1)}_{n}\end{pmatrix}.$$
         Then we see that the second column (and the second row) must miss in $M$ (otherwise $\det(M)=0$). Therefore $\det(M_1)=a^{(1)}_{n+1}\det(M'_1)$ where $M'_1$ is a square submatrix of $\mathcal{}B_n$ giving rise to a principal minor. Since $\lambda a^{(1)}_0\ldots a^{(1)}_n$ is a summand in $\det(M'_1)$, by inductive hypothesis $I=\{2,\ldots,m+1\}$ and 
      $\lambda=1$.

\end{proof}

\section{On the finiteness of the $p$--adic Jacobi--Perron algorithm} \label{sec:fin}

In \cite{MT}, the authors gave some results about the finiteness of the $p$-adic Jacobi--Perron algorithm. We recall these results below.

\begin{proposition} \label{prop:lin-dip}
If the $p$--adic Jacobi--Perron algorithm stops in a finite number of steps when processing the $m$--tuple $(\alpha^{(1)},\ldots , \alpha^{(m)}) \in \mathbb Q_p^m$, then $1,\alpha^{(1)},\ldots , \alpha^{(m)}$ are $\QQ$-linearly dependent. 
\end{proposition}

\begin{proposition} \label{prop:finite}
For an input $(\alpha_0^{(1)},\ldots , \alpha_0^{(m)})\in\QQ_p^m$,  the $p$--adic Jacobi--Perron algorithm terminates in a finite number of steps.
\end{proposition}

Thus, a full characterization of the input vectors which lead to a finite Jacobi--Perron expansion is still missing in the $p$--adic case. 
On the other hand, in the real field it is known that the Jacobi--Perron algorithm stops in a finite number of steps if and only if $1,\alpha^{(1)},\ldots , \alpha^{(m)}$ are $\QQ$-linearly dependent for $m=2$, whereas this is not true for $m \geq 3$, see \cite[Theorem 44]{Sch1} and \cite{Dub, DubA}. Counterexamples in the latter case are provided by $m$-tuples of algebraic numbers belonging to a finite extension of $\QQ$ of degree $<m+1$ and giving rise to a periodic MCF.  This shows that the finiteness an the periodicity  of the Jacobi--Perron algorithm are in some way interrelated.

In this section we shall assume that $1,\alpha^{(1)},\ldots , \alpha^{(m)}$ are linearly dependent over $\QQ$, and associate to every linear dependence relation 
a sequence of integers $(S_n)_{n\geq 0}$, which will  be useful in the investigation of the finiteness of the $p$--adic Jacobi--Perron algorithm. In particular, in the case $m=2$, we shall provide a condition that must be satisfied by the partial quotients of an infinite MCF obtained by the $p$-adic Jacobi--Perron algorithm processing a couple $(\alpha, \beta)$, where $1, \alpha, \beta$ are $\QQ$-linearly dependent.
We shall show in next section that, unlike the real case, even for $m=2$ there exist some input vectors $\bm{\alpha}$ such that $1,\alpha^{(1)},\ldots , \alpha^{(m)}$ are $\QQ$-linearly dependent but their $p$-adic Jacobi--Perron expansion is periodic (and hence not finite). \\
Let us consider $\bm{\alpha}_0=(\alpha^{(1)}_0,\ldots , \alpha^{(m)}_0) \in \mathbb Q_p^m $ and assume that there is a linear dependence relation  
\begin{equation}\label{eq:lindeprel} x_1 \alpha_0^{(1)} + \ldots + x_m \alpha_0^{(m)} + x_{m+1} = 0\end{equation}
with $x_1,\ldots, x_{m+1} \in \mathbb Z$ coprime. Then we can associate to it the  sequence
\begin{equation} \label{eq:s} S_n = x_1 A_{n-1}^{(1)} + \ldots +x_m A_{n-1}^{(m)} + x_{m+1} A^{(m+1)}_{n-1}\end{equation}
for any $n \geq -m$, where $A_n^{(i)}$ are, as usual, the numerators and denominators of the convergents of the MCF of $\bm{\alpha}_0$ defined by \eqref{eq:nd-conv}. It is straightforward to see that the following identities hold:

\begin{align} & S_n \alpha_n^{(1)} + \ldots + S_{n-m+1}\alpha_n^{(m)} + S_{n-m} = 0,  \hbox{ for any $n \geq 0$};\label{eq:uno}\\
 & S_n = a_{n-1}^{(1)} S_{n-1} + \ldots + a_{n-1}^{(m)} S_{n-m} + S_{n-m-1}, \hbox{  for any $n \geq 1$};\label{eq:due} \\
 & S_n = (a_{n-1}^{(1)} - \alpha_{n-1}^{(1)}) S_{n-1} + \ldots + (a_{n-1}^{(m)} - \alpha_{n-1}^{(m)}) S_{n-m}, \hbox{  for any $n \geq 1$};\label{eq:saa}\\
 & S_n = x_1 V_{n-1}^{(1)} + \ldots + x_m V_{n-1}^{(m)}, \hbox{ for any $n \geq -m+1$}.\label{eq:quattro}\\
 &\label{eq:traspostaB}\begin{pmatrix} S_n\\ S_{n-1}\\
\vdots \\ S_{n-m} \end{pmatrix} = \mathcal{B}_{n-1}^T \begin{pmatrix} x_1\\ x_2\\\vdots \\ x_{m+1} \end{pmatrix}
\end{align}
where the superscript $T$ denotes transposition.

\begin{proposition} \label{prop:s}
Given the  sequence $(S_n)_{n \geq -m}$ defined by \eqref{eq:s}, we have that $S_n \in \mathbb Z$, for any $n \geq -m$ and the $\gcd$ of $S_n,\ldots, S_{n-m}$ is a power of $p$. Moreover,
\[|S_n| < \max_{1\leq i \leq m} \{|S_{n-i}|\}\]
so that if the MCF for $(\alpha_0^{(1)},\ldots,\alpha_0^{(m)}) $ is infinite, then  \[\lim_{n \rightarrow +\infty} S_n = 0 \hbox{ in } \QQ_p.\]
\end{proposition}
\begin{proof}
By definition $S_n \in \mathbb Z\left[ \cfrac{1}{p} \right]$, for any $n \geq -m$, and $S_{-m+1}, \ldots S_0 \in \mathbb Z$. Then, using formula \eqref{eq:saa}, 
 and observing that $v_p(a_{n-1}^{(i)} - \alpha_{n-1}^{(i)}) > 0$, for $i = 1, \ldots, m$, where $v_p(\cdot)$ is the $p$-adic valuation, we get $S_n \in \mathbb Z$. The assertion about the $\gcd$ is easily proved by induction, using formula \eqref{eq:due}. 

Since $|a_n^{(i)} - \alpha_n^{(i)}| < 1$, from \eqref{eq:saa}, we have
\[|S_n| \leq \max_{1 \leq i \leq m} \{|a_n^{(i)} - \alpha_n^{(i)}| |S_{n-1}|\} < \max_{1 \leq j \leq m} \{|S_{n-i}|\}.\]
Finally, by Proposition \ref{prop:V} and formula \eqref{eq:quattro} 
we see that $\lim_{n \rightarrow +\infty} S_n = 0$ in $\mathbb Q_p$.
\end{proof}

An immediate consequence of Proposition \ref{prop:s} is the following
\begin{corollary}\label{cor:ennesuemme} For $n\geq 0$, write $n=qm+r$ with $q,r\in\ZZ$ and $0\leq r <m$; then $v_p(S_n)> q$. In particular $v_p(S_n)> \left[ \frac n m\right]$ for every $n\geq 0$.\end{corollary}

Proposition \ref{prop:s} and Corollary \ref{cor:ennesuemme} describe the behaviour of the sequence $(S_n)$ with respect to the $p$-adic norm. We study now its behaviour  with respect to the euclidean norm. We start by a general result.

\begin{proposition} \label{prop:T}
Let $(T_n)_{n \geq -m}$ be any sequence in $\RR$ satisfying
$$T_n = y_n^{(1)} T_{n-1} + \ldots + y_{n}^{(m)} T_{n-m} + T_{n-m-1}, \quad n \geq 1 $$
where $(y_n^{(1)})_{n \geq 1}, \ldots (y_n^{(m)})_{n \geq 1}$ are sequences of elements in $\mathcal{Y}$; then 
$$\lim_{n \rightarrow + \infty} \cfrac{T_n}{p^n} = 0$$
in $\RR$.
\end{proposition}
\begin{proof} 

In the following $\lvert \cdot \rvert_\infty$ stands for the Euclidean norm. We have
\begin{align*}
    \left |\frac{T_n}{p^n} \right |_\infty
    &<\frac 1 2  \left |\frac{T_{n-1}}{p^{n-1}} \right |_\infty + \frac 1 {2p}  \left |\frac{T_{n-2}}{p^{n-2}} \right |_\infty + \ldots +  \frac 1 {2p^{m-1}}  \left |\frac{T_{n-m}}{p^{n-m}} \right |_\infty +  \frac 1 {p^{m+1}}  \left |\frac{T_{n-m-1}}{p^{n-m-1}} \right |_\infty\\
    &\leq K_p \max\left\{\left|\frac{T_{n-1}}{p^{n-1}} \right |_\infty, \left|\frac{T_{n-2}}{p^{n-2}} \right |_\infty, \ldots, \left|\frac{T_{n-m}}{p^{n-m}} \right |_\infty, \left|\frac{T_{n-m-1}}{p^{n-m-1}} \right |_\infty  \right\},
\end{align*}
where $K_p = \cfrac{1}{p^{m+1}} + \cfrac{1}{2} \sum_{k=0}^{m-1} \cfrac{1}{p^k} < 1$. Therefore 
$$ \left |\frac{T_n}{p^n} \right |_\infty< K_p^{n-2} \max\left\{\left|\frac{T_{m}}{p^{m}} \right |_\infty,\left|\frac{T_{m-1}}{p^{m-1}} \right |_\infty, \ldots, \left|\frac{T_{1}}{p} \right |_\infty, \left|T_{0} \right |_\infty  \right\}$$
and the claim follows.
\end{proof}

\begin{corollary} \label{cor:liminftyS}
For the sequence $(S_n)_{n \geq -m+1}$, we have
\[\lim_{n \rightarrow + \infty} \cfrac{S_n}{p^n} = 0\]
in $\mathbb R$.
\end{corollary}
\begin{proof}
By formula \eqref{eq:due}, the sequence $(S_n)_{n \geq -m}$ satisfies the hypothesis of Proposition \ref{prop:T}.
\end{proof}
We shall use the properties stated above to establish some partial converse of Proposition \ref{prop:lin-dip}. The case $(\alpha_0^{(1)},\ldots , \alpha_0^{(m)})\in\QQ^m$ is dealt by Proposition \ref{prop:finite}, so that we can assume that  $(\alpha_0^{(1)},\ldots , \alpha_0^{(m)})\in\QQ_p^m\setminus \QQ^m$.  Notice that in the case $m=2$ under this hypothesis there can be only one linear dependence relation \eqref{eq:lindeprel}, so that the sequence $S_n$ depends only on the sequence  $(\alpha_0^{(1)},\alpha_0^{(2)})\in\QQ_p^2$.
\begin{proposition}\label{prop:denomlimgen} 
Assume that the sequence $\left( \cfrac{S_n}{p^n} \right )$ has bounded denominator, i.e. there exist $k\in \ZZ$ such that $v_p(S_n)\geq n+k$, for every $n$. Then the Jacobi--Perron algorithm stops in finitely many steps when processing the input $(\alpha_0^{(1)},\ldots , \alpha_0^{(m)})$.
\end{proposition}
\begin{proof}
Assume that the Jacobi--Perron algorithm does not stop. Put  $z_n=p^k \cfrac{S_n}{p^n}$; then $z_n\in\ZZ $ and the sequence $(z_n)$ tends to $0$ in the euclidean norm, by Corollary \ref{cor:liminftyS}. It follows that $z_n$ (and hence $S_n$) is 0 for $n\gg 0$, and this is impossible by formula \eqref{eq:uno}. 
\end{proof}
The following theorem is the main result of this section.
To get an uniform notation,  we shall put $\alpha^{(m+1)}_n=a^{(m+1)}_n=1$ for every $n$.
\begin{theorem}\label{teo:finitenessgen}
Assume that $1, \alpha_0^{(1)},\ldots , \alpha_0^{(m)}$ are $\QQ$-linearly dependent and 
\begin{equation}\label{eq:condfin}  v_p(a^{(j)}_n)-v_p(a^{(1)}_n)\geq  j-1\quad\hbox{for $j=3,\ldots, m+1$ and any $n$ sufficiently large.} \end{equation}
 Then the Jacobi--Perron algorithm stops in finitely many steps when processing the input $(\alpha_0^{(1)},\ldots , \alpha_0^{(m)})$.\end{theorem}
 Notice that the condition $v_p(a^{(2)}_n)-v_p(a^{(1)}_n)\geq  1$ is always true by conditions \eqref{eq:conv}.
\begin{proof}
By \eqref{eq:uno} we get
$$\frac{S_n}{p^n}=-\frac{S_{n-1}}{p^{n-1}}\gamma_n^{(1)} -\ldots -\frac{S_{n-m}}{p^{n-m}}\gamma_n^{(m)},$$
where for $j=1,\ldots, m$
$$\gamma_n^{(j)}=\frac {\alpha_n^{(j+1)}} {p^j\alpha_n^{(1)}}.$$
By equations \eqref{eq:conv},  \eqref{eq:alpha-norme} and hypotheses \eqref{eq:condfin} we have $v_p(\gamma_n^{(j)})\geq 0$   
for $n$ sufficiently large. Therefore $v_p\left (\frac{S_n}{p^n}\right ) \geq \min\left\{v_p\left (\frac{S_{n-1}}{p^{n-1}}\right ),\ldots,v_p\left (\frac{S_{n-m}}{p^{n-m}}\right )\right\}$ for $n$ sufficiently large, so that $v_p\left (\frac{S_n}{p^n}\right )\geq K$ for some $K\in\ZZ$. Then we conclude by Proposition \ref{prop:denomlimgen}.
\end{proof}

In the case  $m = 2$, Theorem \ref{teo:finitenessgen} assumes the following simple form.

\begin{corollary}\label{cor:finitenessdimdue} For $m = 2$, if $1, \alpha_0^{(1)}, \alpha_0^{(2)}$ are linearly dependent over $\QQ$ and the $p$-adic Jacobi--Perron algorithm does not stop then $v_p(a_n^{(1)})=-1$ for infinitely many $n\in\NN$.
\end{corollary}

In the next section we shall present some examples where the hypotheses of Corollary \ref{cor:finitenessdimdue} are satisfied.

\begin{remark}
In the classical real case, for $m=2$, it is possible to prove that the Jacobi--Perron algorithm detects rational dependence because the sequences $(V_n^{(1)})$ and $(V_n^{(2)})$ are bounded with respect to the euclidean norm. In fact, this implies that the set of triples $(S_n, S_{n-1}, S_{n-2})$ is finite and consequently the corresponding MCF is finite or periodic. Moreover it is possible to show that a periodic expansion can not occur and consequently the Jacobi--Perron algorithm stops when processes two real numbers $\alpha, \beta$ such that $1, \alpha, \beta$ are $\QQ$-linearly dependent, see \cite{Sch1} for details. 
In the $p$-adic case, the sequences $(V_n^{(i)})$ are  bounded (because they approach zero in $\mathbb Q_p$, see Proposition \ref{prop:V}); but the argument above does not apply, because the $p$-adic norm is non-archimedean. However, considering that $v_p(S_n) > \frac{n}{2}$ by Corollary \ref{cor:ennesuemme}, it could be interesting to focus on the sequence of integers $\left(\cfrac{S_n}{p^{n/2}}\right)$. When this sequence is bounded with respect to the euclidean norm, it is possible to argue  similarly to the real case and deduce the finiteness of the $p$-adic Jacobi--Perron algorithm on the given input.
\end{remark}

\section{On the characteristic polynomial of periodic multidimensional continued fractions} \label{sec:per}

The classical Jacobi--Perron algorithm was introduced over the real numbers with the aim of providing periodic representations for algebraic irrationalities. However, the problem regarding the periodicity of MCFs is still open, since it is not known if every algebraic irrational of degree $m + 1$ belongs to a real input vector of lenght $m$ for which the Jacobi--Perron algorithm is eventually periodic. On the contrary, periodic MCFs have been fully studied over the real numbers. Indeed, it is known that a periodic MCF represents real numbers belonging to an algebraic number field of degree less or equal than $m + 1$, see \cite{Ber2} for a survey on this topic. Moreover, for $m = 2$, Coleman \cite{Col} gave also a criterion for establishing when the periodic MCF converges to cubic irrationalities.
In this section, we start the study of the periodicity of MCFs over $\mathbb Q_p$. In particular, we shall see that, analogously to the real case,  a periodic $p$-adic $m$-dimesional MCF represents algebraic irrationalities of degree less or equal than $m + 1$.

Let us consider a purely periodic MCF of period $N$:

\begin{equation} \label{eq:MCF-period} (\alpha_0^{(1)}, \ldots, \alpha_0^{(m)}) = \left[\left(\overline{a_0^{(1)}, \ldots, a_{N-1}^{(1)}}\right), \ldots, \left(\overline{a_0^{(m)}, \ldots, a_{N-1}^{(m)}}\right)\right],\end{equation}
i.e., $a_{k+N}^{(i)}=a^{(i)}_k$ for every $k\in \NN$ and $i = 1, \ldots, m$. By \eqref{eq:MCF}, we also have $\alpha_{k+N}^{(i)}=\alpha^{(i)}_k$ for every $k\in \NN$ and $i = 1, \ldots, m$, from which, follows
\begin{equation}\label{eq:periouno} \alpha^{(i)}_0=\frac {\alpha^{(1)}_0A^{(i)}_{N-1}+\ldots + \alpha^{(m)}_0A^{(i)}_{N-m}+A^{(i)}_{N-m-1}}{\alpha^{(1)}_0A^{(m+1)}_{n-1}+\ldots + \alpha^{(m)}_0A^{(m+1)}_{n-m}+A^{(m+1)}_{N-m-1}}\end{equation}
using \eqref{eq:alpha0}.
We define the matrix
$$\mathcal{M} := \mathcal{B}_{N-1}=\prod_{j=0}^{N-1}\mathcal{A}_j=\begin{pmatrix} {A^{(1)}_{N-1}} &{A^{(1)}_{N-2}} &\ldots & {A^{(1)}_{N-m-1}}\\
 {A^{(2)}_{N-1}} &{A^{(2)}_{N-2}} &\ldots & {A^{(2)}_{N-m-1}}\\
  \vdots &\vdots&\vdots& \vdots\\
 {A^{(m+1)}_{N-1}} &{A^{(m+1)}_{N-2}} &\ldots & {A^{(m+1)}_{N-m-1}}\end{pmatrix}$$
whose characteristic polynomial $P(X)$ will be also called the characteristic polynomial of the periodic MCF \eqref{eq:MCF-period}.
From equation \eqref{eq:periouno}, we have
\[\mathcal{M}\begin{pmatrix}\alpha^{(1)}_0\\ \vdots\\ \alpha^{(m)}_0\\ 1 \end{pmatrix} = 
 \left (
 \alpha^{(1)}_0 A^{(m+1)}_{N-1} +\alpha^{(2)}_0 A^{(m+1)}_{N-2} +\ldots + \alpha^{(m)}_0 A^{(m+1)}_{N-m}+A^{(m+1)}_{N-m-1} \right )\begin{pmatrix}\alpha^{(1)}_0\\ \vdots\\ \alpha^{(m)}_0\\ 1 \end{pmatrix}.\]
Moreover, by Proposition \ref{prop:sum-prod} we know that $\sum_{i=1}^{m+1}\alpha^{(i)}_N A^{(m+1)}_{N-i} = \alpha_1^{(1)} \cdots \alpha_N^{(1)}$ and, since $\alpha_0^{(1)} = \alpha_N^{(1)}$, we have 
\[\mathcal{M}\begin{pmatrix}\alpha^{(1)}_0\\ \vdots\\ \alpha^{(m)}_0\\ 1 \end{pmatrix} = \alpha^{(1)}_0\ldots \alpha^{(1)}_{N-1} \begin{pmatrix}\alpha^{(1)}\\ \vdots\\ \alpha^{(m)}\\ 1 \end{pmatrix}.\] 
Therefore $\mu :=\alpha^{(1)}_0\ldots \alpha^{(1)}_{N-1}$ is an eigenvalue of $\mathcal{M}$ and a root of the characteristic polynomial $P(X)$.
In the next theorems, we shall see that $\mu$ is the $p$-adic dominant eigenvalue, that is the root greatest in $p$-adic norm of $P(X)$ and that the limits of the periodic MCF \eqref{eq:MCF-period} are strictly related to $\mu$. Note that it is not a loss of generality to consider purely periodic MCFs, since the algebraic properties of the complete quotients of a  MCF coincide with those of the input vector.

\begin{theorem} \label{thm:main}
Given the purely periodic MCF 
\[(\alpha_0^{(1)}, \ldots, \alpha_0^{(m)}) = \left[\left(\overline{a_0^{(1)}, \ldots, a_{N-1}^{(1)}}\right), \ldots, \left(\overline{a_0^{(m)}, \ldots, a_{N-1}^{(m)}}\right)\right]\]
and its characteristic polynomial $P(X)$, then $\mu = \alpha^{(1)}_0\ldots \alpha^{(1)}_{N-1}$ is the greatest root in $p$-adic norm.
\end{theorem}

\begin{proof}

We consider $a_n^{(1)} = \cfrac{\tilde a_n^{(1)}}{p^{k_n}}$, for any $n \geq 0$, where $k_n \geq 0$ ($k_n > 0$, for $n > 0$). We define the quantity $k = k_0 + \ldots + k_{N-1}$ and the matrix
\[ \mathcal{M}':=p^k\mathcal{M}=\mathcal{A}'_0\ldots \mathcal{A}'_{N-1} \]
where
 $$\mathcal{A}'_i=p^{k_i}\mathcal{A}_i=\begin{pmatrix} {\tilde a^{(1)}_{i}} &p^{k_i} &0 &\ldots & 0\\
 p^{k_i}{a^{(2)}_{i}} &0 &p^{k_i} &\ldots & 0\\
  \vdots &\vdots&\vdots& \vdots&\vdots \\
  p^{k_i}a^{(m)}_{i} &0 &0& \ldots & p^{k_i}\\
 p^{k_i} &0 &0  &\ldots & 0\end{pmatrix}\equiv \begin{pmatrix} {\tilde a^{(1)}_{i}} &0 &0 &\ldots & 0\\
 0 &0 &0 &\ldots & 0\\
  \vdots &\vdots&\vdots& \vdots&\vdots \\
  0 &0 &0& \ldots & 0\\
 0 &0 &0  &\ldots & 0\end{pmatrix}\pmod p.$$
Therefore
\begin{equation}\label{eq:emmeprimo} \mathcal{M}'\equiv \begin{pmatrix} {\tilde a^{(1)}_{0}\ldots \tilde a^{(1)}_{N-1}} &0 &0 &\ldots & 0\\
 0 &0 &0 &\ldots & 0\\
  \vdots &\vdots&\vdots& \vdots&\vdots \\
  0 &0 &0& \ldots & 0\\
 0 &0 &0  &\ldots & 0\end{pmatrix}\pmod p.\end{equation}
Let $Q(X)$ be the characteristic polynomial of $\mathcal{M'}$. Then, $\lambda$ is an eigenvalue of $\mathcal{M}$ if and only if  $p^k\lambda $ is an eigenvalue of $\mathcal{M'}$. If $\lambda_1, \ldots, \lambda_{m+1}$ are the eigenvalues of $\mathcal M$, then
\begin{equation*} Q(X) = \prod_{i=1}^{m+1}(X-p^k\lambda_i) = p^{k(m+1)}\prod_{i=1}^{m+1} \left (\frac x {p^{k}}-\lambda_i\right ) = p^{k(m+1)} P\left ( \frac X{p^{k}}\right)
\end{equation*} 
so that
$$P(X)=\frac 1 {p^{k(m+1)}}Q(p^kX).$$
From \eqref{eq:emmeprimo} we have
$$Q(X) \equiv X^m(X-\tilde a_0^{(1)}\ldots \tilde a_{N-1}^{(1)})\pmod p.$$
Thus
$$Q(X)=X^{m+1}+\delta_mX^m+\ldots +\delta_0$$
with
$$\delta_m\equiv \tilde a_0^{(1)}\ldots \tilde a_{N-1}^{(1)}\pmod p, \quad \delta_i\equiv 0\pmod p\hbox{ for } i=0,\ldots, m-1, \quad \delta_0=\pm p^{k(m+1)}.$$
It follows\\
\begin{align*}
         P_\mu(X) &= \frac 1 {p^{k(m+1)}}Q(p^kX)\\
         &= \frac 1 {p^{k(m+1)}}(p^{k(m+1)}X^{m+1}+\delta_mp^{km}X^m+\ldots +\delta_ip^{ki}X^i+\ldots+\delta_0)\\
         &= X^{(m+1)}+\frac {\delta_m}{p^k}X^m+\ldots +\frac {\delta_i}{p^{k(m+1-i)}} X^i+\ldots \pm 1\\
         &=X^{m+1}+\gamma_mX^m+\ldots +\gamma_0,
\end{align*}
where 
$$\gamma_i=\frac {\delta_i}{ p^{k(m+1-i)}}\quad\hbox{ for } i=0,\ldots, m,\quad (\gamma_0=\pm 1).$$
Now, we put $\mu_i=v_p(\delta_i)$for $i=1,\ldots, m$ and observe that
$$\mu_m=0,\quad \mu_i>0\hbox{ for } i=1,\ldots, m-1, \quad \mu_0=k(m+1).$$
We can see that
\begin{align*}
         v_p(\gamma_m)& =v_p(a_0^{(1)}\ldots a_{N-1}^{(1)}) =\sum_{i=0}^{N-1}v_p(a_{i}^{(1)})=-k\\
         v_p(\gamma_i) &=v_p(\delta_i)-k(m+1-i)\\
         &=\mu_i+ik-(m+1)k\quad\hbox{for } i=0,\ldots, m.
\end{align*}
Now we want to study the Newton polygon (see \cite{Gou}) of $P(X)$ for proving that $\mu$ is the root greatest in $p$-adic norm.

The line, in the real plane, passing through the points $(i,v_p(\gamma_i))$ and $(m+1,0)$ has equation
\begin{equation} \label{eq:line} y=\frac {v_p(\gamma_i)}{m+1-i}(-x+m+1), \end{equation}
for any $i = 1, \ldots, m-1$. We will denote with $s_i$ the slope of this line.
From the fact that
$$v_p(\gamma_i)=\mu_i-k(m+1-i)\hbox{ and } \mu_i=v_p(\delta_i)>0,$$
we get
$$\frac {v_p(\gamma_i)}{m+1-i}=\frac {\mu_i}{m+1-i}-k>-k,$$
i.e., the point in the real plane with coordinates $(m,v_p(\gamma_m))=(m,-k)$ lies strictly under the line \eqref{eq:line}, for any $i = 1, \ldots, m-1$.

Thus, the Newton polygon associated to the polynomial $P(X)$ has slopes $(s_1, \ldots, s_m)$, which is a strictly increasing sequence, where the last slope $s_m$ is equal to $k$.
Hence, the claim of the theorem follows from \cite[Theorem 6.4.7]{Gou} and the fact that the sequence $(s_1,\ldots, s_m)$ of slopes is strictly increasing.

\end{proof}

\begin{theorem} \label{thm:munoraz}
Given the purely periodic MCF 
\[(\alpha_0^{(1)}, \ldots, \alpha_0^{(m)}) = \left[\left(\overline{a_0^{(1)}, \ldots, a_{N-1}^{(1)}}\right), \ldots, \left(\overline{a_0^{(m)}, \ldots, a_{N-1}^{(m)}}\right)\right]\]
and its characteristic polynomial $P(X)$, then
\begin{itemize}
\item[a)] $\QQ(\mu)=\QQ(\alpha^{(1)}_0,\ldots,\alpha^{(m)}_0)$
\item[b)] $\mu\not\in\QQ$
\end{itemize}  
where $\mu$ is the greatest root in $p$-adic norm of $P(X)$.
\end{theorem}

\begin{proof} \ \\
$a$) Since $\mu=\alpha^{(1)}_0\ldots \alpha^{(1)}_{N-1}$, obviously
     $\mu\in \QQ(\alpha^{(1)}_0,\ldots,\alpha^{(m)}_0)$; conversely by Theorem \ref{thm:main} the nullspace of $\mathcal{B}_{N-1}-\mu I_{m+1}\in M_{m+1}(\QQ(\mu))$  is $1$-dimensional (where $I_{m+1}$ is the $(m+1) \times (m+1)$ identity matrix and $M_{m+1}(\QQ(\mu))$ denotes the set of $(m+1) \times (m+1)$ matrices with entries in $\QQ(\mu)$).	
Therefore it is generated by a vector $\mathbf{\beta}=(\beta_1,\ldots, \beta_{m+1})$ with entries in $\QQ(\mu)$, which must be proportional to $(\alpha^{(1)}_0,\ldots,\alpha^{(m)}_0,1)$. It follows that $\alpha^{(i)}=\frac {\beta_i}{\beta_{m+1}}\in\QQ(\mu)$ for $i=1,\ldots, m$.\\
$b$) Assume that $\mu\in\QQ$, then $\alpha^{(1)}_0,\ldots,\alpha^{(m)}_0 \in\QQ$. But in this case the MCF corresponding to $(\alpha^{(1)}_0,\ldots,\alpha^{(m)}_0)$ is finite (see \cite{MT}), so that it cannot be periodic.
\end{proof}

From the previous theorems, we have that a periodic MCF converges to a $m$--tuple of algebraic irrationalities of degree less or equal than $m+1$,belonging to the field generated over $\QQ$  by the  the root greatest in $p$-adic norm of the characteristic polynomial. In the characteristic polynomial is irreducible, then the algebraic irrationalities are of the maximum degree. In the following, we see some further properties of the roots of the characteristic polynomial and then we focus on the case $m=2$ for some specific considerations.

\begin{lemma} \label{lemma}
Let $P(X) = X^{m+1} + \gamma_m X^m + \ldots + \gamma_1 X + (-1)^{m(N+1)+1}$ be the characteristic polynomial of the purely periodic MCF $(\alpha_0^{(1)}, \ldots, \alpha_0^{(m)}) = \left[\left(\overline{a_0^{(1)}, \ldots, a_{N-1}^{(1)}}\right), \ldots, \left(\overline{a_0^{(m)}, \ldots, a_{N-1}^{(m)}}\right)\right]$. We have that
\begin{itemize}
\item[a)] every $\gamma_i$ is a polynomial in $\ZZ[a^{(i)}_j, i=1,\ldots , m,\  j=0,\ldots N-1]$ and each monomial has the form $\lambda c_0c_1\ldots c_{N-1}$ where $\lambda\in\ZZ$ and $c_j=1$ or $c_j=a^{(i)}_j$ for some $i=1,\ldots, m$;
\item[b)] the monomial $a_0^{(1)} \cdots a_{N-1}^{(1)}$ appears only in $\gamma_m$.
\end{itemize}
\end{lemma}
\begin{proof} 
Let us observe that any coefficient $\gamma_i$ is the sum of the principal minors of the matrix $\mathcal B_{N-1}$ of order $m+1-i$, for $i = 1, \ldots, m$. Hence the thesis follows from Proposition \ref{prop:minors}. \\
\end{proof}

\begin{theorem} \label{thm:normm1}
Given the purely periodic MCF 
\[(\alpha_0^{(1)}, \ldots, \alpha_0^{(m)}) = \left[\left(\overline{a_0^{(1)}, \ldots, a_{N-1}^{(1)}}\right), \ldots, \left(\overline{a_0^{(m)}, \ldots, a_{N-1}^{(m)}}\right)\right],\]
every root of its characteristic polynomial $P(X) = X^{m+1} + \gamma_m X^m + \ldots + \gamma_1 X + (-1)^{m(N+1)+1}$ has $p$-adic norm less than 1, except for the root greatest in $p$-adic norm $\mu=\alpha_0^{(1)}\cdots \alpha_{N-1}^{(1)}$.
\end{theorem}
\begin{proof}
By Lemma \ref{lemma} and $|a_n^{(1)}|>1$,  $|a_n^{(1)}|>|a_n^{(j)}|$ for $n\in\NN$, $j=2,\ldots , m+1$, we have $|\gamma_i| \leq |a_0^{(1)} \cdots a_{N-1}^{(1)}|$, for any $i = 1, \ldots, m$. Moreover, this inequality becomes equality if and only if $i = m$. If $\lambda_1 = \mu, \lambda_2, \ldots, \lambda_k$ are the roots of $P(X)$ with $p$-adic norm $\geq 1$, then $\gamma_{m+1-k} \geq |\mu| = |a_0^{(1)}\cdots a_{N-1}^{(1)}|$. Recalling that $\gamma_{m+1-k}$ is also the $k$-th elementary symmetric function of the roots, this implies $k = 1$ and the thesis follows.
\end{proof}

\begin{theorem} \label{thm:gersh}
Let $z$ be a complex root of the characteristic polynomial $P(X)$ of the purely periodic MCF $\left[\left(\overline{a_0^{(1)}, \ldots, a_{N-1}^{(1)}}\right), \ldots, \left(\overline{a_0^{(m)}, \ldots, a_{N-1}^{(m)}}\right)\right]$. Then 
	$$|z|_\infty < {p^N} .$$
\end{theorem}
\begin{proof}
By Gershgorin theorem \cite{Gers} there esists a row $j=1,\ldots, m+1$ in $\mathcal{B}_{N-1}$ such that 
$$|z-A^{(j)}_{N-j}|_\infty \leq \sum_{k=1,\ldots, m+1,\  k\not=j }|A^{(j)}_{N-k}|_\infty.$$ In particular 
\[|z|_\infty \leq \sum_{k=1}^{m+1}|A^{(j)}_{N-k}|_\infty < \frac 1 2 \sum_{k=1}^{m+1} p^{N-k+1}\]
by Proposition \ref{prop:boundA}. Moreover,
\[\frac 1 2 \sum_{k=1}^{m+1} p^{N-k+1} = \frac 1 2 p^{N-m}\sum_{k=0}^{m} p^{k} = \frac 1 2 p^{N-m}\frac {p^{m+1}-1}{p-1} \leq p^N.\]
\end{proof}

The previous theorems are useful in order to give some further information about the algebraic properties of the  limits of a periodic MCF. We firstly consider the case $N=1$:
\begin{proposition}
The characteristic polynomial of a purely periodic MCF with period $N=1$ does not have any rational root. In particular when $m=2$ the characteristic polynomial is irreducible over $\QQ$, and the  limits of the MCF generate a cubic field.
\end{proposition}
\begin{proof}
Let $z$ be a rational root  of the characteristic polynomial; by the rational root theorem it must be (up to a sign) a power of $p$. By Theorem \ref{thm:munoraz} $b)$, we know that $z\not=\mu$, and this implies that  $v_p(z) \geq 1$  by Theorem \ref{thm:normm1}. But $|z|_\infty <p $ by Theorem \ref{thm:gersh}, a contradiction.\end{proof}
 In general, a rational root of a MCF with period of length $N$ must satisfy $|z|_\infty < p^N$ and $v_p(z) \geq 1$, so that for the rational root theorem it must be of the kind $\pm p^k$, with $k \leq N-1$. The next proposition gives a necessary condition for the existence of such a root, in the case 
  $m=2$ and $N=2$.

\begin{proposition}
Let us consider the purely periodic MCF $\left[ \left(\overline{a_0, a_1}\right), \left(\overline{b_0,b_1}\right) \right]$. Then its characteristic polynomial $P(X)$ is irreducible over $\QQ$ unless the following condition is verified, possibly interchanging the indices 0 and 1:
\begin{equation}\label{eq:condizirr}
\begin{array}{l} \bullet \hbox{  $a_0$ is of the form $\pm \frac  1 p + w$ with $w\in \ZZ, |w|_\infty\leq \frac{p-1} 2, w\not=0$; and}\\
\bullet \hbox{ either $v_p(a_1p+1)=v_p(a_1)+1 $ (which implies $v_p(b_1)=v_p(a_1)+1, v_p(b_0)=0$)} \\  \hbox{or $a_1$ is of the form $\pm \frac  1 p + u$ with $u\in \ZZ, |u|\leq \frac{p-1} 2, u\not=0$}; \\
\hbox{in the latter case one between $b_0$ and $b_1$ is zero and the other one is equal to $-wu\pm p$.}\end{array}\end{equation}
\end{proposition}
\begin{proof}
Write $$P(X)=X^3+\gamma_2X^2+\gamma_1 X-1,$$
then $$\gamma_2=-(a_0a_1+b_0+b_1),\quad\quad \gamma_1=b_1b_0-a_0-a_1$$
so that
$$P(X)=X(X-b_0)(X-b_1)-(a_0X+1)(a_1X+1).$$
We put $k_1=-v_p(a_1), k_2=-v_p(a_2), k=k_1+k_2$.
By Theorems \ref{thm:normm1} and \ref{thm:gersh} the only possible rational roots of $P(X)$ are $\pm p$. So assume 
\begin{equation}
\label{eq:equazionp} 
P(\pm p)= \pm p(\pm p-b_0)(\pm p-b_1)-(\pm a_0p+1)(\pm a_1p+1)=0.\end{equation}
Notice that the valuation of the first summand is $\geq -k+3$ and that of the second summand is $\geq -k+2$.  Therefore, the valuation of the second summand must be $\geq -k+3$. This implies that at least one between $a_0$ and $a_1$, say $a_0$,  must satisfy $v(\pm a_0p+1)>-k_0+1$, that is $a_0p\equiv \mp 1\pmod p$. Since $a_0\in\mathcal{Y}$ this implies $a_0=\mp\frac  1 p+w$ with $w\in \ZZ, |w|_\infty \leq \frac{p-1} 2$ and \eqref{eq:equazionp} becomes 
\begin{equation}
\label{eq:equazionp2} 
\pm (\pm p-b_0)(\pm p-b_1)-w(\pm a_1p+1)=0.\end{equation}
We show that $w\not=0$: otherwise one between $b_0$ and $b_1$ should be equal to $\pm p$, which is  a contradiction because $b_0,b_1\in\mathcal{Y}$. \\
The right-hand side of \eqref{eq:equazionp2} has valuation $\geq -k_1+1$; and $v_p(\pm p-b_0)\geq 0$, $v_p(\pm p-b_1)\geq -k_1+1$. If the valuation of the right side is exactly  $-k_1+1$ then it must be $v(b_0)=0, v(b_1)=-k_1+1$. On the other hand, if  the valuation of the right side is $>-k_1+1$ then $a_1p\equiv \mp 1\pmod p$. As above  this implies $a_0=\mp\frac  1 p+u$ with $u\in \ZZ, |u|_\infty\leq \frac{p-1} 2, u\not=0$ and \eqref{eq:equazionp2} becomes
\begin{equation}
\label{eq:equazionp3} 
\pm (\pm p-b_0)(\pm p-b_1)-wup=0.\end{equation} This implies that one between $b_0$ and $b_1$ is $0$, the other one (say $b_i$) has valuation $0$,  and satisfies $\pm p-b_i=wu$.
\end{proof}
In order to provide numerical examples, the following proposition will be useful. 
\begin{proposition}\label{prop:trovaABC}
Let us consider the purely periodic 2-dimensional MCF $(\alpha, \beta) = \left[ \left(\overline{a_0, \ldots, a_{N-1}}\right), \left(\overline{b_0, \ldots, b_{N-1}}\right) \right]$ and suppose that its characteristic polynomial $P(X)$ is reducible. Let $z=\pm p^k$ be the (unique) rational root of $P(X)$, then  the $1$-dimensional eigenspace $\mathcal{L}\subseteq \QQ^3$ of the transpose of $\mathcal{B}_{N-1}$ associated to $z$ coincides with the space $\mathcal{L'}$ of rational vectors $(x,y,z)$ such that $x\alpha+y\beta+z=0$.
\end{proposition}
\begin{proof} Notice firstly that the space $\mathcal{L'}$ is one-dimensional, because  Theorem \ref{thm:munoraz} and the reducibility of $P(X)$ imply that $[\QQ(\alpha,\beta):\QQ]=2$. Therefore there is a linear dependence relation
$$x_1\alpha +x_2\beta +x_3=0$$
with coprime $x_1,x_2,x_3\in\ZZ$, and $(x_1,x_2,x_3)$ generates $\mathcal{L'}$.
Since $\alpha_N=\alpha   $, $\beta_N=\beta$,  by  property \eqref{eq:uno} of the sequence $(S_n)$ defined by \eqref{eq:s},  the vector $(S_N, S_{N-1}, S_{N-2})$ must be a rational multiple of $(x_1,x_2,x_3)$; then by \eqref{eq:quattro}  $(x_0,y_0,z_0)$ is an eigenvector associated to a rational eigenvalue, so that it belongs to $\mathcal{L}$. 
\end{proof}

\begin{example}
     Condition \eqref{eq:condizirr} is essential. Consider the following examples.

     \begin{itemize}
     \item For $p=5$, the periodic MCF $(\alpha, \beta) = \left[\left(\overline{\frac 4 5, \frac  {11}5}\right), \left(\overline{1, 2}\right)\right]$ has characteristic polynomial 
     $$P(X)=X^3-\frac{119}{25} X^2-X-1 $$
     and
     $$P(X)=(X-5)\left (X^2-\frac 6 {25}X+\frac 1 5\right ).$$
    Moreover by using Proposition \ref{prop:trovaABC} we find the linear dependence relation between $\alpha, \beta$ and $1$:
     $$20\alpha + 5\beta + 4=0.$$
     \item For $p=3$, the periodic MCF $(\alpha, \beta) = \left[\left(\overline{\frac 2 3, \frac  {5}3}\right), \left(\overline{1, 0}\right)\right]$ has characteristic polynomial
     \begin{equation*} P(X) =X^3-\frac{19}{9} X^2-\frac 7 3 X-1=(X-3)\left (X^2+\frac 8 {9}X+\frac 1 3\right ).\end{equation*}
and
     $$6\alpha+3\beta +2=0.$$
     
     \item For $p=3$, the periodic MCF $(\alpha, \beta) = \left[\left(\overline{\frac 2 3, \frac  {13}9}\right), \left(\overline{1, \frac{1}{3}}\right)\right]$ has characteristic polynomial
     \begin{equation*} P(X) =X^3-\frac{62}{27} X^2-\frac {16}{9} X-1 =(X-3)\left (X^2+\frac {19} {27}X+\frac 1 3\right ).\end{equation*}
		The linear dependence relation between $\alpha, \beta, 1$ is the same as in the previous case:
     $$6\alpha + 3\beta + 2=0.$$
     \end{itemize}
\end{example}     
The above examples also show  $\QQ$-linearly dependent numbers having a periodic (hence not finite) expansion by the $p$-adic Jacobi--Perron algorithm.\\
     At the present time we were not able to find examples of $m$-tuples of $\QQ$-linearly dependent $p$-adic numbers whose MCF is infinite and not periodic. Therefore, we state the following 
     
     \begin{conjecture} Let $\bm{\alpha}= (\alpha^{(1)}, \ldots, \alpha^{(m)})\in \QQ_p^m$ be such that $1, \alpha^{(1)},\ldots, \alpha^{(m)}$ are $\QQ$-linearly dependent. Then the $p$-adic Jacobi-Perron algorithm for $\bm{\alpha}$ is finite or periodic.
     \end{conjecture}
  \section*{Acknowledgments}
  We thank Matteo Semplice for stimulating conversations.


\begin{thebibliography}{99}

\bibitem{Ass} Assaf, S., Chen, L. C., Cheslack-Postava, T., Cooper, B., Diesl, A., Garrity, T., Lepinski M., Schuyler, A., A dual approach to triangle sequences: A multidimensional continued fraction algorithm, Integers 5(A08) (2005) 39 pp.

\bibitem{Bea} Beaver, O. R., Garrity, T. A two-dimensional Minkowski $?(x)$ function, J. Number Theory 107 (2004) 105-134.

\bibitem{Bed1} Bedocchi, E., A note on p--adic continued fractions, Ann. Mat. Pura Appl., 152, (1988), 197--207.

\bibitem{Bed2} Bedocchi, E., Sur le développement de $\sqrt{m}$ en fraction continue p--adique, Manuscripta Math., 67, (1990), 187--195.

\bibitem{Ber} Bernstein, L., New infinite classes of periodic Jacobi-Perron algorithms, Pacific J. Math. 16(3) (1965) 439–469.

\bibitem{Ber2} Bernstein, L., The Jacobi--Perron algorithm -- its theory and application, Lecture Notes in Mathematics, 207, Springer, Berlin, (1971).

\bibitem{Bro1} Browkin, J., Continued fractions in local fields I, Demonstratio Mathematica, 11, (1978), 67--82.

\bibitem{Bro2} Browkin, J., Continued fractions in local fields II, Mathematics of Computation, 70, (2000), 1281--1292.

\bibitem{Cap} Capuano, L., Veneziano, F., Zannier, U., An effective criterion for periodicity of l--adic continued fractions, Mathematics of Computation, Ready Online, DOI: https://doi.org/10.1090/mcom/3385, (2018).

\bibitem{Col} J. B. Coleman, A test for the type of irrationality represented by a periodic ternary continued fraction, American Journal of Mathematics, 52, (1930), 835--842.

\bibitem{Dub} Dubois, E., Paysant--Le Roux, R., Developpements péeriodiques par l'algorithme de Jacobi--Perron et nombres de Pisot--Vijayaraghavan, C. R. Acad. Sci. Paris Ser. A-B., 272, (1971), 649--652.

\bibitem{DubA} Dubois, E., Etude des interruptions dans l'algorithme de Jacobi--Perron, Bullettin of the Australian Mathematical Society, 69, (2004), 241--254.

\bibitem{Gar} Garrity, T., On periodic sequences for algebraic numbers, J. Number Theory, 88, (2001), 86--103.

\bibitem{Ger} German, O. N., Lakshtanov, E. L., On a multidimensional generalization of Lagrange's theorem for continued fractions, Izv. Math., 72, (2008), 47--61.

\bibitem{Gers} Gerschgorin, S., \"Uber die Abgrenzung der Eigenwerte einer Matrix, Izv. Akad. Nauk. USSR Otd. Fiz.-Mat. Nauk, 7, (1931), 749--754.

\bibitem{Gou} Gouvea, F. Q., p-adic numbers - An introduction, Second Edition, Springer-Verlag Berlin Heidelberg, (1997).

\bibitem{Han} Hancl, J., Jassova, A., Lertchoosakul, P., Nair, R., On the metric theory of p--adic continued fractions, Indagationes Mathematicae, 24, (2013), 42--56.

\bibitem{Hen} Hendy, M. D., Jeans, N. S., The Jacobi--Perron algorithm in integer form, Math. Comput., 36, (1981), 565--574.

\bibitem{Her} Hermite, C., Extraits de letters de M. Ch. Hermite a M. jacobi sur differents objets de la theorie des nombres, J. Reine Angew. Math, 40, (1850), 286.

\bibitem{Jac} Jacobi, C. G. J., Ges. werke, VI, Berlin Academy, (1891), 385--426.

\bibitem{Kar} Karpenkov, O., Constructing multidimensional continued fractions in the sense of Klein, Math. Comput., 78, (2009), 1687--1711.

\bibitem{Lao} Laohakosol, V., A characterization of rational numbers by p--adic Ruban continued fractions, J. Austral. Math. Soc. Ser. A, 39,(1985), 300--305.

\bibitem{Mil} Miller, J., On p--adic continued fractions and quadratic irrationals, Ph.D. Thesis, The University of Arizona, 2007.

\bibitem{Mur1} Murru, N., On the periodic writing of cubic irrationals and a generalization of Rédei functions, International Journal of Number Theory, 11(3), (2015), 779--799.

\bibitem{Mur2} Murru, N., Linear recurrence sequences and periodicity of multidimensional continued fractions, The Ramanujan Journal, 44, (2017), 115--124.

\bibitem{MT} Murru, N., Terracini, L., On p--adic multidimensional continued fractions, Preprint, Available at https://arxiv.org/abs/1805.00072, (2018).

\bibitem{Oot} Ooto, T., Transcendental p--adic continued fractions, Math. Z., 287, (2017), 1053--1064.

\bibitem{Per} Perron, O., Grundlagen fur eine theorie des Jacobischen kettenbruch algorithmus, Math. Ann., 64, (1907), 1--76.

\bibitem{Poo} van der Poorten, A. J., Schneider's continued fractions, Number theory with an emphasis on the Markoff spectrum, Lecture Notes in Pure and Appl. Math., 147, Dekker, New York, (1993), 271--281.

\bibitem{Rub} Ruban, A., Certain metric properties of the p--adic numbers, Sibirsk Math. Z., 11, (1970), 222--227.

\bibitem{Schn} Schneider, T., Uber p--adische Kettenbruche, Symposia Mathematica, 4, (1968/69), 181--189.

\bibitem{Sch1} Schweiger, F., Multidimensonal continued fractions, Oxford University Express, Oxford, (2000).

\bibitem{Sch2} Schweiger, F., Brun meets Selmer, Integers 13 (2013), Paper A17, 10 pp.

\bibitem{Tam} Tamura, J., Yasutomi, S., A new multidimensional continued fraction algorithm, Math. Comput., 78, (2009), 2209--2222.

\bibitem{Til} Tilborghs, F., Periodic p--adic continued fractions, A Quarterly Journal of Pure and Applied Mathematics, 64, (1990), 383--390.

\bibitem{Weg} de Weger, B. M. M. , Periodicity of p--adic continued fractions, Elemente der Math., 43, (1988), 112--116.

\end{thebibliography}
\end{document}